\documentclass{lms-nodoi}

\usepackage{graphicx}
\usepackage{amsfonts,amsmath,amssymb,amscd}
\usepackage[all,cmtip]{xy}
\usepackage[backref]{hyperref}
\usepackage{verbatim}
\usepackage{enumerate}
\usepackage{color}
\usepackage{pinlabel}

\definecolor{darkg}{RGB}{0,128,0}

\newtheorem{mainthm}{Theorem}

\newtheorem{thm}{Theorem}[section]
\newtheorem{lma}[thm]{Lemma}
\newtheorem{cor}[thm]{Corollary}
\newtheorem{prp}[thm]{Proposition}

\newnumbered{rmk}{Remark}

\newnumbered{dfn}{Definition}
\newnumbered{exm}{Example}
\newnumbered{quest}{Question}

\newcommand{\CC}{\mathbf{C}}

\newcommand{\RR}{\mathbf{R}}
\newcommand{\ZZ}{\mathbf{Z}}
\newcommand{\OP}{\mathrm}

\newcommand{\co}{\mskip0.5mu\colon\thinspace}
\newcommand{\id}{\mathrm{id}}

\author{Georgios Dimitroglou Rizell and Jonathan David Evans}

\extraline{The first author was partially supported by the ERC Starting Grant of Fr\'{e}d\'{e}ric Bourgeois StG-239781-ContactMath.}

\title[Exotic spheres, symplectomorphism groups]{Exotic spheres and the topology\\ of symplectomorphism groups}

\begin{document}
\maketitle
\begin{abstract}
We show that, for certain families $\phi_{\mathbf{s}}$ of diffeomorphisms of high-dimensional spheres, the commutator of the Dehn twist along the zero-section of $T^*S^n$ with the family of pullbacks $\phi^*_{\mathbf{s}}$ gives a noncontractible family of compactly-supported symplectomorphisms. In particular, we find examples: where the Dehn twist along a parametrised Lagrangian sphere depends up to Hamiltonian isotopy on its parametrisation; where the symplectomorphism group is not simply-connected, and where the symplectomorphism group does not have the homotopy-type of a finite CW-complex. We show that these phenomena persist for Dehn twists along the standard matching spheres of the $A_m$-Milnor fibre. The nontriviality is detected by considering the action of symplectomorphisms on the space of {\em parametrised} Lagrangian submanifolds. We find related examples of symplectic mapping classes for $T^*(S^n\times S^1)$ and of an exotic symplectic structure on $T^*(S^n\times S^1)$ standard at infinity.
\end{abstract}

\section{Introduction}

Arnol'd \cite{Arnold} first noticed that the monodromy of the Lefschetz fibration
\[(z_0,\ldots,z_n)\to\sum_{i=0}^nz_i^2\]
around the unit circle is naturally realised by a compactly-supported symplectomorphism $\tau$ of the fibre, $T^*S^{n}$, called the {\em model Dehn twist}. Given a Lagrangian embedding $\ell\co S^n\to X$ of a sphere into a symplectic manifold $X$, we can implant the Dehn twist into a Weinstein neighbourhood and obtain, canonically up to contractible choices, a symplectomorphism $\tau_{\ell}$ in the group $\OP{Symp}^c(X)$ of compactly-supported symplectomorphisms. Symplectomorphisms obtained this way are called {\em Dehn twists}; in this paper we will see that the isotopy class of $\tau_{\ell}$ can depend on the parametrisation $\ell$.

Since their introduction by Arnol'd, Dehn twists have been studied extensively by Seidel \cite{SeidelTS2,GradedLag,SeidelLES,SeidelLectures}; they provide one of the few known sources of symplectomorphisms which are nontrivial in the symplectic mapping class group (the group of components of the symplectomorphism group). Indeed, in some low-dimensional cases \cite{Evans,Gromov,Lick,SeidelTS2,Wu} their isotopy classes are known to generate the symplectic mapping class group.

Given a diffeomorphism $\phi\co S^n\to S^n$ one gets a non-compactly supported symplectomorphism $\phi^*\co T^*S^n\to T^*S^n$ by pulling back covectors. Conjugating the model twist by $\phi^*$ gives a new symplectomorphism $\tau_{\phi}:=\phi^*\tau(\phi^{-1})^*$ which is again compactly supported; if $\ell$ is a Lagrangian embedding and we implant $\tau_{\phi}$ into a Weinstein neighbourhood instead of $\tau$, the result is precisely $\tau_{\ell\circ\phi}$. One can therefore use the topology of the diffeomorphism group of $S^n$ as a potential source for topology in $\OP{Symp}^c(X)$ for any $X$ containing a Lagrangian sphere.

\begin{dfn}
Let $\Theta_n$ denote the group of smooth homotopy $n$-spheres under connected sum. Fix a cotangent fibre $\Lambda \subset T^*S^n$ and let $\mathcal{L}_{n} \subset \Theta_n$ denote the subset of homotopy spheres which admit a Lagrangian embedding into $T^*S^n$, with the additional requirement that the embedding intersects $\Lambda$ transversely in exactly one point.
\end{dfn}

It follows from work of Abouzaid \cite{Abouzaid}, Abouzaid-Kragh \cite{AbouzaidKragh} and Ekholm-Smith \cite{EkholmSmith} that:

\begin{prp}[(See Section \ref{sct:localknot} for the proof)]\label{prp:localknot}
$\mathcal{L}_{n}\subset bP_{n+1}$, where $bP_{n+1}$ denotes the set of smooth homotopy $n$-spheres which bound parallelisable manifolds.
\end{prp}
\begin{rmk}
\label{rmk:poincare}
In the case when $n=1,2,3,5,6,7$ the above proposition is vacuously true, as follows by the solution to the generalised Poincar\'{e} conjecture in these dimensions: such a homotopy sphere is diffeomorphic to $S^n$. For a fuller discussion of the difference between $\Theta_n$ and $bP_{n+1}$, see \cite{KervaireMilnor}. The sizes of $\Theta_n$ and $bP_{n+1}$ are listed respectively as A001676 and A187595 in the Online Encyclopedia of Integer Sequences \cite{OEIS}.
\end{rmk}

\begin{dfn}
Given a $k$-parameter family $\phi_{\mathbf{s}}\in\OP{Diff}(S^n)$ of diffeomorphisms ($\mathbf{s}\in S^k$) representing a homotopy class $\alpha\in\pi_k(\OP{Diff}(S^n))$, define the homotopy $(n+k+1)$-sphere
\begin{gather*}
S^{n+k+1}_{\alpha}:=\left(D^{n+1}\times S^k\right)\cup_{\Phi}\left(S^n\times D^{k+1}\right),\\
\Phi\co S^n\times S^k\to S^n\times S^k,\ \Phi(x,y)=(\phi_{y}(x),y).
\end{gather*}
In particular, this construction provides a homomorphism $\lambda^{n+k}_{k,k}\co\pi_k(\OP{Diff}(S^n))\to\Theta_{n+k+1}$. Given that $n > 5$, it is well-known that $\lambda^n_{0,0}$ is an isomorphism $\pi_0(\OP{Diff}(S^n))\to\Theta_{n+1}$: surjectivity is shown in {\cite[Chapter 9]{Milnor}}, injectivity follows from Cerf's theorem \cite{Cerf} that $\OP{Diff}^+(D^{n+1})$ is connected (otherwise two gluing diffeomorphisms $\phi_1$ and $\phi_2$ could differ by a nontrivial mapping class which nontheless extends over the ball, defining a diffeomorphism of $S^{n+1}_{\phi_1}$ with $S^{n+1}_{\phi_2}$). Also, $\lambda^{n+1}_{1,1}$ is surjective: the map $\pi_1(\OP{Diff}(S^n))\to\Theta_{n+2}\cong\pi_0(\OP{Diff}(S^{n+1}))$ factors through the inclusion isomorphism $\pi_0(\OP{Diff}(D^{n+1},\partial D^{n+1}))\stackrel{\cong}{\to}\pi_0(\OP{Diff}(S^{n+1}))$ via the connecting homomorphism $\pi_1(\OP{Diff}(S^n))\to\pi_0(\OP{Diff}(D^{n+1},\partial D^{n+1}))$ of the fibration
\[\OP{Diff}(D^{n+1},\partial D^{n+1})\to\OP{Diff}(D^{n+1})\to\OP{Diff}(S^n)\]
and the cokernel of this connecting homomorphism is zero by the homotopy long exact sequence of the fibration, since $\pi_0(\OP{Diff}(D^{n+1}))=0$ (again by Cerf's theorem).

The filtration of $\Theta_N$ by the images of these homomorphisms is called the Gromoll filtration. There are many results on the nontriviality of the Gromoll filtration, for example {\cite[Theorem 7.4]{BurgheleaLashofII}}, \cite{AntonelliEtAl,CrowleySchick}. The image of $\lambda^{n+k}_{k,k}$ is usually written $\Gamma^{n+k+1}_{k+1}\subset\Theta_{n+k+1}$.
\end{dfn}

\begin{mainthm}\label{thm:maintheorem}
Let $\phi_{\mathbf{s}}\in\OP{Diff}(S^n)$ be a $k$-parameter family of diffeomorphisms ($\mathbf{s}\in S^k$) representing a homotopy class $\alpha\in\pi_k(\OP{Diff}(S^n))$ such that $S^{n+k+1}_{\alpha}\not\in\mathcal{L}_{n+k+1}$. Then:
\begin{enumerate}
\item the family $\tau^{-1}\tau_{\phi_{\mathbf{s}}}$ is nontrivial in $\pi_k(\OP{Symp}^c(T^*S^n))$;
\item moreover if $\ell\co S^n\to A^n_m$ is a Lagrangian embedding of one of the standard matching spheres in the complex $n$-dimensional $A_m$-Milnor fibre then the family $\tau_{\ell}^{-1}\tau_{\ell\circ\phi_{\mathbf{s}}}$ is nontrivial in $\pi_k(\OP{Symp}^c(A^n_m))$.
\end{enumerate}
\end{mainthm}
\begin{rmk}
A result of Seidel shows that the Dehn twist has infinite order inside $\pi_0(\OP{Symp}^c(A^n_m))$, see e.g.~\cite[Lemma 5.7]{GradedLag} for a proof. The same proof also shows that the elements $\tau_{\ell}^{-1}\tau_{\ell \circ \phi}$ in $\pi_0(\OP{Symp}^c(A^n_m))$ are not Hamiltonian isotopic to any non-zero power of a (reparametrised) Dehn twist. The reason is that these elements act trivially on the grading of the Lagrangian sphere $\ell$ (viewed as a graded Lagrangian submanifold) while, whenever $n>1$, any non-zero power of the Dehn twist acts non-trivially on this grading.
\end{rmk}

By Proposition \ref{prp:localknot}, any element of $\alpha\in\pi_k(\OP{Diff}(S^n))$ with $S^{n+k+1}_{\alpha}\not\in bP_{n+k+1}$ is a class to which Theorem \ref{thm:maintheorem} can be applied. For example, whenever $\Theta_N\setminus bP_{N+1}$ is nonempty, surjectivity of $\lambda^{N-1}_{0,0}$ and $\lambda^{N-1}_{1,1}$ imply that we get nontrivial elements of $\pi_0(\OP{Symp}^c(T^*S^{N-1}))$ and $\pi_1(\OP{Symp}^c(T^*S^{N-2}))$. Another nice example comes from the fact \cite[Table in Section 3]{AntonelliEtAl} that $\Gamma^{13}_3=\Theta_{13}=\ZZ/3$, $bP_{14}=0$. This implies that $\pi_2(\OP{Symp}^c(T^*S^{10}))\neq 0$ and since the identity component $\OP{Ham}^c(T^*S^n)$ is a path-connected H-space, {\cite[Theorem 6.11]{Browder}} gives the following corollary:

\begin{cor}\label{cor:browder}
The group $\OP{Ham}^c(T^*S^{10})$ does not have the homotopy-type of a finite CW complex.
\end{cor}

Our methods extend to prove:

\begin{mainthm}\label{thm:anothersymp}
If there is a nontrivial homotopy class $\alpha\in\pi_1(\OP{Diff}(S^n))$ with the property that $S^{n+2}_{\alpha}\not\in\mathcal{L}_{n+2}$, then there is a symplectomorphism $\Psi\in\OP{Symp}^c(T^*(S^n\times S^1))$ which is not isotopic to the identity through compactly supported symplectomorphisms.
\end{mainthm}

Given a loop representing $\alpha$ we give an explicit construction of $\Psi$. The isotopy class of $\Psi$ depends only on the homotopy class $\alpha$.

\begin{mainthm}\label{thm:exoticform}
If $n\equiv 3\mod 4$, $n\geq 7$ and $\mathcal{L}_{n+1} \subsetneq \Theta_{n+1}$, there is a symplectic form on $T^*(S^n\times S^1)$ which is standard at infinity, homotopic to the standard form through nondegenerate two-forms standard at infinity, but not symplectomorphic to the standard form via a compactly-supported diffeomorphism.
\end{mainthm}

Finally, regarding the smooth isotopy class, one can say the following.

\begin{prp}\label{prp:smoothisotopy}
Let $\ell\geq 1$ and let $\phi\in\OP{Diff}(S^{4\ell+3})$ be a diffeomorphism. The symplectomorphisms $\tau$ and $\tau_{\phi}$ of $T^*S^{4\ell+3}$ are compactly supported smoothly isotopic.
\end{prp}
\begin{proof}
Let $\iota$ be a parametrisation of the zero-section and $\phi$ be a diffeomorphism. Work of Haefliger \cite{Plongements} shows that, when $n\geq 5$, there is a smooth isotopy $\iota_t$ between $\iota_0=\iota$ and $\iota_1=\iota\circ\phi$. The differentials $D\iota_t$ give a path of bundle maps which, we would like to prove, is homotopic to a path of Lagrangian bundle maps. Once we know this, the Weinstein framing of the normal bundle, which identifies it with the cotangent bundle canonically up to contractible choices, is carried along this smooth isotopy. Since the construction of the Dehn twist only depends on the choice of an identification of the normal bundle of $S^n$ with $T^*S^n$, this implies that the Dehn twists associated to $\iota$ and $\iota\circ\phi$ are smoothly isotopic.

To show that $D\iota_t$ is homotopic to a path of Lagrangian bundle maps, we first show that $\iota_0$ and $\iota_1$ are homotopic through Lagrangian immersions. It follows from the h-principle for Lagrangian immersions \cite{Lee}, \cite{Gro} that $\iota_0$ and $\iota_1$ are homotopic through Lagrangian immersions if a difference class $d(\iota,\iota\circ\phi)\in\pi_{4\ell+3}(U(4\ell+3))=\ZZ$ vanishes. Fixing $\iota$, the assignment $\phi\mapsto d(\iota,\iota\circ\phi)$ is a homomorphism
\[\pi_0(\OP{Diff}(S^{4\ell+3})) \to \pi_{4\ell+3}(U(4\ell+3)).\]
Since the former group is finite by \cite{KervaireMilnor}, this difference class is zero. Therefore there is a path $\iota_t'$ of Lagrangian immersions between $\iota_0$ and $\iota_1$.

The Hirsch-Smale h-principle for immersions \cite{Hirsch}, \cite{Smale} shows that the fundamental group of the space of immersions $S^n\to T^*S^n$ is $\pi_{n+1}(V_{2n,n})$, where $V_{2n,n}=SO(2n)/SO(n)$ is a Stiefel manifold. This group vanishes when $n=4\ell+3$ by {\cite[Table (b)]{Paechter1}}. Therefore the paths $\iota_t$ and $\iota_t'$ are homotopic. This implies that $D\iota_t$ is homotopic to a path of Lagrangian bundle maps.
\end{proof}

\subsection{Relation to Seidel's work}

This paper is related to, but orthogonal to, Seidel's paper \cite{SeidelExotic} where he constructs symplectomorphisms of cotangent bundles of exotic spheres analogous to iterated Dehn twists (although it is currently unknown if these constructions can be made to have compact support). In particular, the $k=0$ case of Theorem \ref{thm:maintheorem} answers a question posed in that paper. The idea we use, which amounts to detecting nontrivial symplectic mapping classes by looking for a change in the parametrisation of a cotangent fibre, was inspired by Keating's argument in Section 5 of that paper.

\section{Preliminaries}

\subsection{Lagrangian suspension}\label{sec:lagsusp}

The main tool we use is the {\em Lagrangian suspension} construction \cite{AudinCob} which allows us to turn a Hamiltonian isotopy $\phi^t(L)$ of a Lagrangian submanifold into a Lagrangian embedding $L\times[0,1]\to X\times T^*[0,1]$.

We need a generalisation of this construction to $k$-parameter families $\psi_{\mathbf{q}}$, $\mathbf{q}\in I^k$ where $\psi_{\mathbf{q}}$ is the identity for $\mathbf{q}\in\partial(I^k)\setminus(I^{k-1}\times\{1\})$ and $\psi_{\mathbf{q}}$ preserves $L$ for all $\mathbf{q}\in\partial(I^k)$. We call the resulting $L\times I^k\subset X\times T^*I^k$ the {\em Lagrangian suspension of $L$ along $\psi_{\mathbf{q}}$}. The case $k=0$ is the usual Lagrangian suspension.

In the following we let $I=[0,1]$ denote the unit interval and identify $T^*[0,1]$ with $\RR\times[0,1]$. We write $\theta_{I^k}$ for the canonical 1-form on $T^*I^k$. We will only need the case when $(X,\omega)$ is a simply-connected symplectic manifold. We can therefore omit the requirements of being Hamiltonian in the following section.

\begin{prp}
Suppose we have a family $\psi_{\mathbf{q}}$ of symplectomorphisms of a simply-connected symplectic manifold $X$, parametrised by $\mathbf{q}\in I^k$, for which there is a nonempty open subset $U\subset X$ such that $\psi_{\mathbf{q}}|_U=\OP{id}$ for all $\mathbf{q}$. Then there exists a symplectomorphism
\begin{gather*}
\Psi \colon (X \times T^*I^k,\omega \oplus d\theta_{I^k}) \to (X \times T^*I^k,\omega \oplus d\theta_{I^k}),\\
(x,\mathbf{q},\mathbf{p}) \mapsto (\psi_{\mathbf{q}}(x),\mathbf{q},\mathbf{p}-\mathbf{H}(x,\mathbf{q})),
\end{gather*}
for a smooth function
\[\mathbf{H} \colon X \times I^k \to \RR^k\]
which may be taken to vanish on $U \times I^k$. We call $\Psi$ the suspension of $\psi_{\mathbf{q}}$.

Moreover, if $\psi_{\mathbf{q}}(L)=L$ for a Lagrangian submanifold $L \subset (X,\omega)$ and all $\mathbf{q}\in A$ for a subset $A \subset I^k$, it follows that $\mathbf{H}$ is constant along each subset $L \times \{\mathbf{q}\}$ with $\mathbf{q} \in A$.
\end{prp}
\begin{proof}
The smooth map
\[(x,\mathbf{q},\mathbf{p}) \mapsto (\psi_{\mathbf{q}}(x),\mathbf{q},\mathbf{p})\]
is unfortunately not symplectic, since the pull-back of $\omega \oplus d\theta_{I^k}$ is of the form
\begin{gather*}
\omega \oplus d\theta_{I^k} - \eta, \\
\eta:=\sum_{i=1}^k dq_i \wedge \alpha_i(x,\mathbf{q})+\sum_{i,j=1}^k f_{i,j}(x,\mathbf{q})dq_i \wedge dq_j.
\end{gather*}
Here $\alpha_i$ is a $k$-parameter family of one-forms on $X$ and $f_{i,j} \co X \times I^k \to \RR$ are functions satisfying $f_{i,j}=-f_{j,i}$. The fact that both the forms $\omega \oplus d\theta_{I^k} - \eta$ and $\omega \oplus d\theta_{I^k}$ are symplectic, and hence that $\eta$ is closed, implies that $d_x \alpha_i=0$, where $d_x$ denotes the exterior derivative on $X$. Consequently, since $X$ is assumed to be simply-connected, $\alpha_i=d_x H_i(x,\mathbf{q})$ for some smooth function
\[\mathbf{H}=(H_1,\hdots,H_k) \colon X \times I^k \to \RR^k\]
uniquely determined by the requirement that it vanishes on $U \times I^k$. Here we have used the fact that $\eta$, and hence $\alpha_i$, vanishes on $U \times I^k$.

Closedness moreover implies that
\[d_x(f_{i,j}-f_{j,i})=\partial_{q_j} \alpha_i(x,\mathbf{q})-\partial_{q_i} \alpha_j(x,\mathbf{q})=\partial_{q_j} d_x H_i-\partial_{q_i} d_x H_j.\]
Changing the order of the partial derivations, we obtain that
\[f_{i,j}-f_{j,i}=\partial_{q_j}H_i-\partial_{q_i}H_j+F_{i,j}(\mathbf{q}),\]
where $F_{i,j}(\mathbf{q})$ vanishes since $\eta$, and hence $f_{i,j}$, vanishes on $U \times I^k$. It is now easy to check that the modified diffeomorphism $\Psi$ preserves the symplectic form.

Suppose that $L \subset (X,\omega)$ is a Lagrangian submanifold preserved set-wise by $\psi_{\mathbf{q}_0}$.  The Lagrangian condition implies that $\eta$ vanishes along $L \times \{ \mathbf{q}_0\}$ and, hence, so do the one-forms $\alpha_i$ when pulled back to $L \times \{\mathbf{q}_0\} \subset X \times\{\mathbf{q}_0\}$. It follows that $\mathbf{H}$ is constant along $L \times \{ \mathbf{q}_0\}$.
\end{proof}

\begin{rmk}\label{rmk:const}
Note that one can replace $\mathbf{H}$ by $\mathbf{H}(x,\mathbf{q})+\mathbf{F}(\mathbf{q})$ for an arbitrary function $\mathbf{F}(\mathbf{q})$ and one still gets a symplectomorphism. If $\psi_{\mathbf{q}}(L)=L$ then $\mathbf{H}$ is constantly equal to $\mathbf{C}(\mathbf{q})$ along $L\times\{\mathbf{q}\}$. We can then choose $\mathbf{F}(\mathbf{q})=-\mathbf{C}(\mathbf{q})$ to ensure that $\Psi(L,\mathbf{q},0)=(L,\mathbf{q},0)$.
\end{rmk}

\begin{dfn}[Lagrangian suspension]
Let $\iota\co L\to X$ be a Lagrangian embedding and $\psi_{\mathbf{q}}$ be a family of symplectomorphisms parametrised by $\mathbf{q}\in I^k$. Suppose that $\psi_{\mathbf{q}}(L)=L$ for all $\mathbf{q}$ in a neighbourhood of $\partial I^k$. Then the Lagrangian
\[\Sigma L:= \Psi(L\times I^k)\subset X\times T^*I^k\]
is called the {\em Lagrangian suspension} of $L$ along $\psi_{\mathbf{q}}$. By Remark \ref{rmk:const} one can ensure that $\Psi(L)$ agrees with $L\times I^k$ in a neighbourhood of the boundary.
\end{dfn}

\begin{rmk}\label{rmk:param}
We will further assume that $\psi_{\mathbf{q}}=\id$ for $\mathbf{q}\in\partial I^k\setminus(I^{k-1} \times \{1\})$ and that the map $I^{k-1} \times \{1\} \to\OP{Diff}(L)$ which sends $\mathbf{q}$ to $\psi_{\mathbf{q}}|_L$ represents a given homotopy class $\alpha\in\pi_{k-1}(\OP{Diff}(L))$.
\end{rmk}

\subsection{Open symplectic embeddings in $T^*S^N$}\label{sec:openemb}

Let
\[A^n_m:=\left\{\prod_{j=0}^m(x_0-j)+x_1^2+\cdots+x_n^2=0\right\}\subset\CC^{n+1}.\]
denote the complex $n$-dimensional $A_m$-Milnor fibre. In the proofs of the main results we will continually need to appeal to the existence of certain open symplectic embeddings of open subsets of $A^n_m\times T^*I^{k+1}$ into $T^*S^{n+k+1}$. In this section we will establish the result we need.

Consider the Lefschetz fibration
\[\pi^n_m\colon A^n_m\to\CC,\quad\pi^n_m(x_0,\ldots,x_n)=x_0\]
with $m+1$ critical points at $0,1,\ldots,m$. Given an arc $\gamma$ in $\CC$ connecting two critical points, denote by $L_{\gamma}$ the Lagrangian matching cycle fibring over $\gamma$. We write $\gamma_i\colon[0,1]\to\CC$, $i=1,\ldots,m$, for the arc $\gamma_i(t)=i+t-1$ (considered as a matching path for $\pi^n_m$). For brevity we will write $L_i$ for the matching sphere $L_{\gamma_i}\subset A^n_m$.

\begin{prp}\label{prp:openemb}
For all $k,n\geq 0$ there is an open symplectic embedding of an open neighbourhood $W$ of $\bigcup_{i=1}^mL_i\times I^{k+1}\subset A^n_m\times T^*I^{k+1}$ into $T^*S^{n+k+1}$ such that:
\begin{itemize}
\item $L_1\times I^{k+1}$ is sent to a subset of the zero-section $S^{n+k+1}$;
\item the image of the embedding is disjoint from a particular cotangent fibre $\Lambda\subset T^*S^{n+k+1}$.
\end{itemize}
\end{prp}

\begin{proof}
The proposition follows from the case $k=0$ because there is an open symplectic embedding of $T^*S^{n+1}\times T^*I^k$ into $T^*S^{n+k+1}$: take the left-inverse of the pullback along the open inclusion $S^{n+1}\times I^k\to S^{n+k+1}$ of a tubular neighbourhood of a sphere $S^{n+1}\subset S^{n+k+1}$. Note that the image of this embedding is disjoint from $\Lambda=T^*_xS^{n+k+1}$ for any $x$ not contained in the tubular neighbourhood of $S^{n+1}$.

\par {\em Case $k=0$.} The manifold $T^*S^{n+1}$ admits a sequence of Lefschetz fibrations $p_m$, $m=1,2,3\ldots$, where:
\begin{itemize}
 \item $p_1=\pi^{n+1}_1\colon A^{n+1}_1=T^*S^{n+1}\to\CC$ is the standard Lefschetz fibration with two singular fibres, whose general fibre is $A^n_1$.
\item $p_m$ is obtained from $p_1$ by stabilising $m-1$ times: the smooth fibre is $A_m^n$ and there are $m+1$ singular fibres living over the critical points $0,1,\ldots,m$.
\end{itemize}
The vanishing cycles associated to the $m+1$ critical points of the fibration $p_m$ are
\[L_1,\ L_1,\ L_2,\ \cdots \ L_{m-1},\ L_m\subset A^n_m.\]
Moreover $\gamma_1$ defines a matching path for $\pi_m^{n+1}$ whose matching sphere is the zero-section in $T^*S^{n+1}$. See Figure \ref{fig:leffib}.

\begin{figure}[htb]
\begin{center}
\labellist
\pinlabel $\color{black}{A^n_2}$ at 270 230
\pinlabel $\color{red}{L_1}$ at 125 200
\pinlabel $\color{black}{L_2}$ at 125 175
\pinlabel $\color{red}{S^{n+1}}$ at 210 220
\endlabellist
\includegraphics[width=200px]{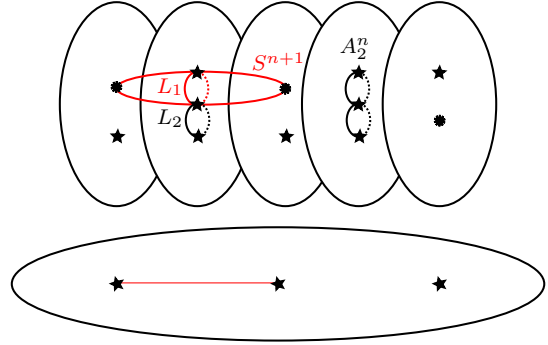}
\caption{The Lefschetz fibration $p_2$ on $T^*S^{n+1}$.}
\label{fig:leffib}
\end{center}
\end{figure}

By symplectic parallel transport, we can try to trivialise $p_m$ over a compact subset $D^*I=\{x+iy\in\CC\ :\ x\in[\epsilon,1-\epsilon],\ |y|\leq 1\}\subset\CC$. Since the fibres are noncompact then, {\em a priori}, we are only able to achieve this over some compact neighbourhood $W'$ of the skeleton $\bigcup_{i=1}^nL_i\subset A^n_m$. This gives us a symplectic embedding of $W=W'\times D^*I$ into $T^*S^{n+1}$ and by construction $L_1\times I$ is identified with a subset of the zero-section (the matching sphere for $\gamma_1$).

There is a cotangent fibre of $T^*S^{n+1}$ which arises as the Lefschetz thimble for $p_m$ living over the ray $(-\infty,0]$, and which hence is disjoint from $p_m^{-1}(D^*I)$.
\end{proof}

\subsection{Proof of Proposition \ref{prp:localknot}}\label{sct:localknot}

The proposition will follow by combining several previously known results concerning the smooth structure of certain exact Lagrangian homotopy spheres. The first result in this direction is due to Abouzaid, who in \cite{Abouzaid} showed that 
\begin{thm}\label{thm:abouzaid}
If $\Sigma\to T^*S^n$ is an exact Lagrangian embedding of a homotopy $n$-sphere, and if $n\geq 9$, $n\equiv 1\mod 4$, then $\Sigma$ bounds a parallelisable manifold.
\end{thm}

Recall that the Whitney immersion is an exact Lagrangian immersion $w\co S^n\looparrowright\CC^n$ with one transverse double point and that the stable Lagrangian Gauss map of a Lagrangian immersion $\Sigma\looparrowright\CC^n$ is a map $\Sigma\to U/O$ which defines an element of $\pi_n(U/O)$ if $\Sigma$ is a homotopy sphere. In any dimension $n \ge 1$, Ekholm-Smith \cite{EkholmSmith} prove that:
\begin{thm}\label{thm:ekholmsmith}
If $\Sigma\looparrowright\CC^n$ is an exact Lagrangian immersion of a homotopy sphere having one transverse double point and satisfying the conditions that:
\begin{itemize}
\item the Legendrian contact homology grading of the double-point of $\Sigma$ coincides with that of the Whitney immersion (i.e.~it is two), and
\item the stable Lagrangian Gauss map of $\Sigma$ is homotopic to that of the Whitney immersion (i.e.~it is null homotopic),
\end{itemize}
then $\Sigma$ bounds a parallelisable manifold.
\end{thm}

\begin{proof}[of Proposition \ref{prp:localknot}]
When $n \geq 9$, $n\equiv 1\mod 4$, the result follows a fortiori by Theorem \ref{thm:abouzaid}. We are thus left with the cases $n=4$, and $n \ge 8$ satisfying $n \not\equiv 1 \mod 4$ (see Remark \ref{rmk:poincare}), for which the goal is to apply Theorem \ref{thm:ekholmsmith}.

Assume that $\Sigma\in\mathcal{L}_{n}$, let $\Lambda$ be a cotangent fibre in $T^*S^n$, and let $\ell\co\Sigma\to T^*S^n$ be a Lagrangian embedding which intersects $\Lambda$ once transversely. Without loss of generality we can assume that $\ell$ coincides with the zero section in a neighbourhood $U$ of $\Lambda$, since this can be achieved by a Hamiltonian isotopy. We can then push $\ell$ forward into an immersed Weinstein neighbourhood $f \co D^*S^n \looparrowright \CC^n$ of the Whitney immersion to obtain a Lagrangian immersion $\tilde{w}:=f\circ \ell \co\Sigma\looparrowright\CC^n$. Assuming that the preimages in $S^n$ of the double point lie in the neighbourhood $U$, the immersion we construct has precisely one double point, the grading of which moreover coincides with that of the Whitney immersion.

In all dimensions $n\equiv 0\mod 2$ or $n\equiv 7\mod 8$ the stable Lagrangian Gauss map is nullhomotopic because $\pi_n(U/O)=0$. Therefore, $\Sigma$ bounds a parallelisable manifold by Theorem \ref{thm:ekholmsmith}.

We are left with the case $n\equiv 3\mod 8$. In this case we have $\pi_n(U)=\ZZ$ and $\pi_n(U/O)=\ZZ/2$. If $\Sigma$ is a homotopy sphere then the isotopy classes of exact Lagrangian immersions $\ell\colon \Sigma\looparrowright T^*S^n$ are in bijection with $\pi_n(U)=\ZZ$; let us write $m(\ell)$ for the integer associated with $\ell$. Abouzaid and Kragh show that if $\ell\colon \Sigma\to T^*S^n$ is an exact Lagrangian embedding and $n\equiv 3\mod 8$ then $m(\ell)$ is even. In more detail, they construct a sequence of homomorphisms {\cite[Eq 2.16]{AbouzaidKragh}}
\[\ZZ=\pi_n(U)\stackrel{\cong}{\to}\pi_{n-1}(\ZZ\times BU)\to\pi_{n-1}(\ZZ\times BO)\to\pi_{n-1}(BH)=\ZZ/2\]
where $H$ is the stable group of self-homotopy equivalences of spheres and show that (a) when $n\equiv 3\mod 8$ the composite map $\ZZ\to\ZZ/2$ is just reduction modulo 2 {\cite[Table 1, $k=2$, last line]{AbouzaidKragh}} and (b) if $\ell$ is an embedding then $m(\ell)$ is in the kernel of this composition {\cite[Proposition 2.2]{AbouzaidKragh}}.

Now, as before, we consider the Weinstein neighbourhood $f\colon D^*S^n\looparrowright\CC^n$ of the Whitney immersion and take the composition $f\circ\ell$ to get a Lagrangian immersion of $\Sigma$ in $\CC^n$. The homotopy class of the stable Lagrangian Gauss map of $\Sigma\looparrowright\CC^n$ is just $m(\ell)\mod 2$ in $\pi_n(U/O)=\ZZ/2$. From the previous paragraph, we deduce that if $\ell$ is an embedding then the homotopy class of the stable Gauss map of $f\circ \ell$ is zero, and hence, by Theorem \ref{thm:ekholmsmith}, $\Sigma$ bounds a parallelisable manifold.
\end{proof}

\section{Proof of Theorem \ref{thm:maintheorem}}\label{sec:maintheorem}

Let $\phi_{\mathbf{s}}\co S^n\to S^n$ be a family of diffeomorphisms depending on a parameter $\mathbf{s}\in I^k$ representing the homotopy class $\alpha\in\pi_k(\OP{Diff}(S^n))$. We first show that we do not lose any generality if we assume that there is a disc $V\subset S^n$ for which $\phi_{\mathbf{s}}|_V=\OP{id}$. Note that we may need to modify the homotopy class $\alpha$ in order to ensure this, but that we change neither the diffeomorphism type of $S^{n+k+1}_{\alpha}$ nor the homotopy class of $\tau^{-1}\tau_{\phi_{\mathbf{s}}}$ in $\pi_k(\OP{Symp}(T^*S^{n+k+1}))$ in the process (see Remarks \ref{rmk:nochangeone} and \ref{rmk:nochangetwo}). Next, we prove a lemma describing how the conjugated Dehn twist acts on a (para\-me\-tr\-ised) cotangent fibre. We then prove Theorem \ref{thm:maintheorem}.

\subsection{Modifying the homotopy class}

\begin{lma}\label{lma:modifyhtpyclass}
Let $\alpha\in\pi_k(\OP{Diff}(S^n))$ be a homotopy class represented by a family $\phi_{\mathbf{s}}\co S^n\to S^n$ of diffeomorphisms depending on a parameter $\mathbf{s}\in I^k$ such that for $\mathbf{s}\in\partial(I^k)$, $\phi_{\mathbf{s}}=\id$. Let $V\subset S^n$ be an open ball. There exists a homotopy class $\alpha'\in\pi_k(\OP{Diff}(S^n))$ represented by a family $\psi_{\mathbf{s}}$ of diffeomorphisms such that
\begin{itemize}
\item $\psi_{\mathbf{s}}=\id$ for $\mathbf{s}\in\partial(I^k)$,
\item $\psi_{\mathbf{s}}|_V=\id_V$ for all $\mathbf{s}\in I^k$,
\item $\psi_{\mathbf{s}}$ is homotopic rel $\partial(I^k)$ to $\phi_{\mathbf{s}}\circ r_{\mathbf{s}}$ for some family of rotations $r_{\mathbf{s}}\in SO(n+1)$.
\end{itemize}
\end{lma}
\begin{proof}
Let $P\to S^n$ denote the principal $GL^+(n,\RR)$-frame bundle of $S^n$. Fix a point $x\in S^n$ and an orthonormal frame $p\in P_x$. There is an evaluation map $\OP{ev}\co\OP{Diff}(S^n)\to P$ which sends $\phi$ to $\phi_*p$. The space $P$ is homotopy equivalent to the subset of orthonormal frames and the composite map $F\co SO(n+1)\subset\OP{Diff}(S^n)\stackrel{\OP{ev}}{\to} P$ is a homotopy equivalence which is one-to-one onto its image, $P_O$, the subspace of orthonormal frames.

Let $f\co S^k\to \OP{Diff}(S^n)$ be a map representing the homotopy class $\alpha\in\pi_k(\OP{Diff}(S^n))$. Then $\OP{ev}\circ f\co S^k\to P$ is homotopic to a map $f_O\co S^k\to P_O$. Since $\OP{ev}$ is a homotopy fibration, the homotopy between $f$ and $f_O$ lifts to a homotopy in $\OP{Diff}(S^n)$ so there is another map $g\co S^k\to\OP{Diff}(S^n)$ representing $\alpha$ and satifying $\OP{ev}(g(S^k))\subset P_O$. Finally, define $r_{\mathbf{s}}^{-1}\in SO(n+1)$ to be $F^{-1}(\OP{ev}(g(\mathbf{s})))$. This is an $S^k$ of rotations defining some class $\beta\in\pi_k(\OP{Diff}(S^n))$ and having the properties that:
\begin{itemize}
\item each $\psi_{\mathbf{s}}:=\phi_{\mathbf{s}}\circ r_{\mathbf{s}}$ fixes the point $x$ and the frame $p$ at $x$, and
\item the map $\mathbf{s}\mapsto\psi_{\mathbf{s}}$ represents the homotopy class $\alpha'=\alpha\cdot\beta^{-1}\in\pi_k(\OP{Diff}(S^n))$.
\end{itemize}
In an exponential chart at $x$, since the frame $p$ is fixed by the derivative of $\psi_{\mathbf{s}}$, one can straighten the diffeomorphisms $\psi_{\mathbf{s}}$, canonically up to contractible choices, so that they fix not only the frame $p$ but a small neighbourhood $V$ of $x$.
\end{proof}

\begin{rmk}\label{rmk:nochangeone}
The classes $\alpha$ and $\alpha'$ differ by the class $\beta$ in $\pi_k(\OP{Diff}(S^n))$, but the rotations $r_{\mathbf{s}}$ extend to the disc $D^{n+1}$. This means that $S^{n+k+1}_{\alpha}$ and $S^{n+k+1}_{\alpha'}$ are diffeomorphic.
\end{rmk}
\begin{rmk}\label{rmk:nochangetwo}
The model Dehn twist commutes with pullback along isometries because it is defined in terms of the geodesic flow for the round metric. Therefore if $r$ is a rotation and $\ell\co S^n\to X$ is a Lagrangian embedding then $\tau_{\ell\circ\phi\circ r}=\tau_{\ell\circ\phi}$.
\end{rmk}

\begin{cor}\label{cor:wlog}
In the proof of Theorem \ref{thm:maintheorem}, we can assume without loss of generality that $\alpha$ is represented by a family of diffeomorphisms $\phi_{\mathbf{s}}$ which are the identity on some ball $V\subset S^n$.
\end{cor}

\subsection{Action on a cotangent fibre}

The family of diffeomorphisms in the next lemma is assumed to have the property that each $\phi_{\mathbf{s}}$ is supported in an open ball $U$. By Corollary \ref{cor:wlog}, this is allowed without loss of generality.

\begin{lma}\label{lma:param}
Let $\iota_x\co\RR^n\to T^*S^n$ be a parametrisation of the cotangent fibre $T_xS^n$ and let $\phi_{\mathbf{s}}$, $\mathbf{s} \in I^k$, be a $k$-parameter family of diffeomorphisms of $S^n$ supported in a ball $U\subset S^n\setminus\{x\}$. When $k \ge 1$ we require this family to be trivial over $\partial I^k$. Then the family
\[C_{\mathbf{s}}=\tau^{-1}\phi_{\mathbf{s}}^*\tau(\phi_{\mathbf{s}}^{-1})^*\]
is homotopic in $\OP{Symp}^c(T^*S^n)$ to a family $C'_{\mathbf{s}}$ with the property that
\begin{itemize}
\item the symplectomorphism $C'_{\mathbf{s}}$ preserves $\iota_x(\RR^n)$, and 
\item the diffeomorphism $\iota_x^{-1}\circ C'_{\mathbf{s}}\circ\iota_x\co\RR^n\to\RR^n$ is a family of compactly-supported diffeomorphisms which, when extended to $S^n=\RR^n\cup\{\infty\}$, is isotopic to the family $\phi_{\mathbf{s}}$.
\end{itemize}
\end{lma}
\begin{proof}
Let $L_1$ denote $\iota_x(\RR^n)$ and let $U^*\subset T^*S^n$ denote the union of cotangent fibres over $U\subset S^n$, so $U^*$ contains the support of $\phi_{\mathbf{s}}$ for all $\mathbf{s}$. Note that there is a compactly-supported Hamiltonian path $\psi_t$ with the properties that
\begin{itemize}
\item $\psi_0=\OP{id}$, and
\item the intersection between $\psi_1\circ\tau(L_1)$ and $U^*$ equals $U$. See Figure \ref{fig:support}.
\end{itemize}

\begin{figure}[htb]
\begin{center}
\labellist
\pinlabel $W\subset\:T^*S^n$ at 180 240
\pinlabel $\color{blue}{L_1\cap W=}$ at 107 80
\pinlabel $\color{blue}{T^*_xS^n\cap W}$ at 108 60
\pinlabel $\color{red}{\tau(L_1)}$ at 95 185
\pinlabel $\color{red}{\psi_1\tau(L_1)}$ at 310 185
\pinlabel $S^n$ at 177 120
\pinlabel $U$ at 238 74
\pinlabel $\color{darkg}{U^*}$ at 42 220
\pinlabel $\color{darkg}{U^*}$ at 250 220
\endlabellist
\includegraphics[height=150px]{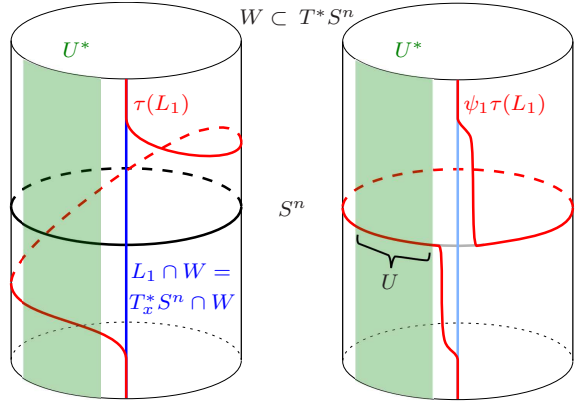}
\caption{The cotangent bundle $T^*S^n$, where $L_1$ is the cotangent fibre $T_xS^n$ for a point $x$ in the complement of $U^*$ and $S^n$ is the zero-section. After an isotopy $\psi_t$ we can assume that $\psi_1(\tau(L_1))\cap U^*=U$, where $U$ is the subset of the zero-section where all the diffeomorphisms $\phi_{\mathbf{s}}$ are supported. It follows that $C'_{\mathbf{s}}(L_1)=L_1$ setwise.}
\label{fig:support}
\end{center}
\end{figure}

We will set $C^t_{\mathbf{s}}=\tau^{-1}\psi_t^{-1}\phi_{\mathbf{s}}^*\psi_t\tau(\phi_{\mathbf{s}}^{-1})^*$. This gives a homotopy between the families $C_{\mathbf{s}}=C^0_{\mathbf{s}}$ and $C'_{\mathbf{s}}:=C^1_{\mathbf{s}}$. We need to show that $C'_{\mathbf{s}}$, defined this way, has the desired properties.

First, the domain of $(\phi^{-1})^*$ is disjoint from $L_1$, so $(\phi^{-1})^*$ fixes $L_1$ pointwise.

Second, we apply $\psi_1\circ\tau$; we have $\psi_1(\tau(L_1))\cap U^*=U$ and $U$ is contained in the zero-section.

Third, $\phi^*$ fixes $\psi_1(\tau(L_1))$ setwise because $\psi_1(\tau(L_1))\cap U^*=U$, which is obviously is fixed setwise by $\phi^*$.

Fourth, $\tau^{-1}(\psi_1)^{-1}$ sends
\[\phi^*(\psi_1\tau(L_1))=\psi_1(\tau(L_1))\]
back to $L_1$. Thus we see that each symplectomorphism $C^1_{\mathbf{s}}$ fixes $L_1$ setwise.

Post-composing $C'_{\mathbf{s}}\circ\iota_x$ with $\iota_x^{-1}$ gives a family of diffeomorphisms of $\RR^n$. At the third stage, the parametrisation changed by $\phi_{\mathbf{s}}$, so $\iota_x^{-1}\circ C'_{\mathbf{s}}\circ\iota_x$ is homotopic (as a family of diffeomorphisms of $S^n=\RR^n\cup\{\infty\}$) to $\phi_{\mathbf{s}}$.
\end{proof}

\subsection{Proof of Theorem \ref{thm:maintheorem}}

Let $\ell\colon S^n\to A^n_m$ be a Lagrangian embedding of one of the matching spheres in the $A_m$-Milnor fibre. Without loss of generality, we assume $m\geq 2$: if the theorem holds for $m\geq 2$ then in particular $\tau_{\ell}^{-1}\circ\tau_{\ell\circ\phi_{\mathbf{s}}}$ cannot be nullhomotopic inside $\OP{Symp}^c(A^n_1)$ or we could perform this nullhomotopy in a Weinstein neighbourhood of any sphere in $A^n_m$, $m\geq 2$. Recall that the matching sphere living over the path $[k-1,k]\subset\CC$ is called $L_k$; we will assume that $\ell(S^n)=L_2$ and we do not lose any generality since the matching spheres are all conjugate under the group of compactly-supported symplectomorphisms.

Suppose that $\phi_{\mathbf{s}}$, $\mathbf{s}=(s_1,\ldots,s_k)\in I^k$, is a $k$-parameter family of diffeomorphisms of $S^n$ such that $\phi_{\mathbf{s}}=\id$ for $\mathbf{s}$ in a neighbourhood of $\partial (I^k)$ and for which there exists a ball $U\subset S^n$ for which $\phi_{\mathbf{s}}|_U=\OP{id}$ for all $\mathbf{s}$. This gives rise to a $k$-parameter family of symplectomorphisms
\[C_{\mathbf{s}}\co=\tau_{\ell}^{-1}\tau_{\ell\circ\phi_{(\mathbf{s},1)}}\co A^n_m\to A^n_m\]

Without loss of generality, the ball $\ell(U)$ can be chosen to contain the intersection point $L_1\cap L_2$, where $L_1$ is the matching sphere living over the path $[0,1]\subset\CC$. The embedding $\ell$ extends to an embedding $\varpi$ of a Weinstein neighbourhood $\mho\subset T^*S^n$ into $A^n_m$; we can assume that there is a point $x\not\in U$ such that $\varpi(T_x^*S^n\cap \mho)\subset L_1$. Rescaling so that the model Dehn twist $\tau$ has support inside $\mho$, we can apply Lemma \ref{lma:param} to $\mho$ and deduce that the family $C_{\mathbf{s}}$ is homotopic to a family $C'_{\mathbf{s}}$ which preserves $L_1$ setwise and such that $C'_{\mathbf{s}}|_{L_1}$ is homotopic inside $\OP{Diff}(L_1)$ to $\phi_{\mathbf{s}}$.

We will assume that $\mathbf{s}\mapsto C'_{\mathbf{s}}$ is nullhomotopic in $\OP{Symp}^c(T^*A^n_m)$, and derive a contradiction to the assumption that $S^{n+k+1}_{\alpha}\not\in\mathcal{L}_{n+k+1}$. Suppose $C'_{(\mathbf{s},t)}$ is a nullhomotopy of $C'_{\mathbf{s}}$ satisfying $C'_{(\mathbf{s},t)}=\id$ for all $(\mathbf{s},t)\in\partial(I^{k+1})\setminus(I^k\times\{1\})$, and $C'_{(\mathbf{s},1)}=C'_{\mathbf{s}}$.

By Proposition \ref{prp:openemb}, we have an open symplectic embedding $e\co W\to T^*S^{n+k+1}$ where $W$ is a Weinstein neighbourhood of $\bigcup_{i=1}^nL_i\times I^{k+1}\subset A^n_m\times T^*I^{k+1}$. Taking the Lagrangian suspension of $L_1\subset A^n_m$ along $C'_{(\mathbf{s},t)}$ and pushing inwards along a suitable Liouville flow on $A^n_m\times T^*I^{k+1}$ preserving the product structure, we obtain a Lagrangian embedding $L_1\times I^{k+1}\to W$ which agrees with the standard embedding near the boundary.

\begin{figure}[htb]
\begin{center}
\includegraphics[width=200px]{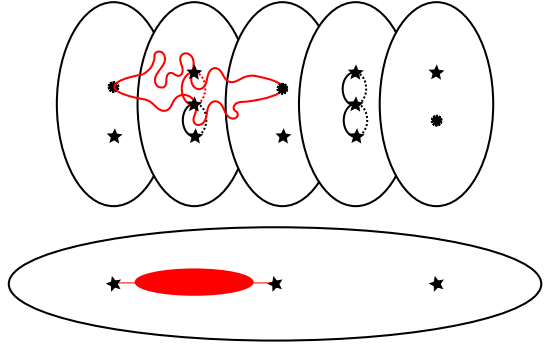}
\caption{$k=0$ case: Lagrangian suspension of a Hamiltonian isotopy of $C'$ with the identity would allow us to construct an exotic Lagrangian sphere in $T^*S^{n+1}$ (compare with Figure \ref{fig:leffib}).}
\label{fig:exotic}
\end{center}
\end{figure}

Replacing $e(L_1\times I^{k+1})$ with this Lagrangian suspension gives us a Lagrangian submanifold of $T^*S^{n+k+1}$ diffeomorphic to $S^{n+k+1}_{\alpha}$ (see Figure \ref{fig:exotic} for an illustration in the case $k=0$). This Lagrangian submanifold intersects a cotangent fibre once transversely, by the second property of the embedding $e$ from Proposition \ref{prp:openemb}, so $S^{n+k+1}_{\alpha}\in\mathcal{L}_{n+k+1}$. Since we assumed that $S^{n+k+1}_{\alpha}\not\in\mathcal{L}_{n+k+1}$ we have a contradiction, hence $\mathbf{s}\mapsto C'_{\mathbf{s}}$ is not nullhomotopic in $\OP{Symp}^c(T^*A^n_m)$.

\section{Proof of Theorem \ref{thm:anothersymp}}
\label{sec:symp}
Let $\phi_t$ be a loop of diffeomorphisms of $S^n$ with $\phi_0=\phi_1=\OP{id}$. We now describe how the loop of compactly-supported symplectomorphisms $\tau^{-1} \tau_{\phi_t}$ of $T^*S^n$ can induce a symplectomorphism $\Psi \in \mathrm{Symp}^c(T^*(S^n \times S^1))$ which is not isotopic to the identity through compactly supported symplectomorphisms.

By Lemma \ref{lma:modifyhtpyclass} we can assume without loss of generality that there is a ball $V\subset S^n$ such that each diffeomorphism $\phi_t$ fixes $V$ pointwise. Let $x$ be a point in $V$.

By Lemma \ref{lma:param}, we can homotope the loop $\tau^{-1}\tau_{\phi_t}$ of compactly supported symplectomorphisms to a loop $C'_t$ of compactly supported symplectomorphisms which preserve the cotangent fibre $T^*_xS^n$ setwise. We can write $C'_t=\phi_H^t$, where $H_t \colon T^*S^n \to \RR$ is a time-dependent Hamiltonian with compact support. The suspension
\begin{gather*}
\Phi_H \colon T^*S^n \times T^*(0,1) \to T^*S^n \times T^*(0,1),\\
(x,q,p) \mapsto (\phi^q_H(x),q,p-H_q(\phi^q_H(x))),
\end{gather*}
is not compactly-supported, but descends along the projection
\begin{gather*}
T^*S^n \times T^*(0,1) \to T^*S^n \times (0,1) \times \RR/\ZZ, \\
(x,q,p) \mapsto (x,q,[p]),
\end{gather*}
to a compactly supported symplectomorphism $\Psi \in \mathrm{Symp}^c(T^*S^n \times (0,1) \times \RR/\ZZ)$, where we have endowed $(0,1) \times \RR/\ZZ \subset T^*S^1$ with the standard symplectic form.
\addtocounter{mainthm}{-2}
\begin{mainthm}
If the loop $\phi_t$ represents a class $\alpha\in\pi_1(\OP{Diff}(S^n))$ for which  $S^{n+2}_{\alpha}\not\in\mathcal{L}_{n+2}$ then $\Psi \in \mathrm{Symp}^c(T^*(S^n \times S^1))$ is not isotopic to the identity through compactly supported symplectomorphisms.
\end{mainthm}
\begin{proof}
The proof is similar to the proof of Theorem \ref{thm:maintheorem} ($k=1$ case). If $\Psi$ were isotopic to the identity in some compact region (say contained inside $D^*S^n\times (0,1)\times S^1$) then we could lift it to the universal cover $D^*S^n\times I\times\RR=D^*S^n \times T^*I$ and get an isotopy of $\Phi_H$ with the identity fixing a neighbourhood of the boundary
\[\partial(D^*S^n) \times T^*I \: \cup \: D^*S^n \times T^*_{\partial I}I.\]
Because $T^*_xS^n$ is fixed setwise by each $C'_t$, the suspension of $T^*_xS^n \times I \subset T^*S^n \times T^*I$ under this isotopy is a Lagrangian embedding $e\co\RR^n\times I^2\hookrightarrow T^*S^n \times (T^*I)^2$ which may be assumed to agree with $T_x^*S^n\times I^2$ outside of $D_x^*S^n \times I^2$, as well as in a neighbourhood of $T^*S^n\times\partial(I^2)$ (see Remark \ref{rmk:const}).

There is an embedding $\iota\co D^*S^n\to A^n_2$ of a disc subbundle into $A^n_2$ satisfying $\iota(S^n)=L_2$ and $\iota(D^*_xS^n)=V\subset L_1$. Note that
\[L'=(\iota\times\OP{id})(e(\RR^n\times I^2) \cap (D^*S^n\times T^*(I^2)))\]
agrees with $V\times I^2$ near the boundary.

Using Proposition \ref{prp:openemb} we can embed an open neighbourhood of $(L_1\cup L_2)\times I^2\subset A^n_2\times T^*(I^2)$ into $T^*S^{n+2}$ so that $L_1\times I^2$ is mapped into a subset of $S^{n+2}$. Excising $V\times I^2\subset S^{n+2}$ and replacing it with $L'$ gives a Lagrangian $S^{n+2}_{\alpha}\subset T^*S^{n+2}$ which intersects a cotangent fibre once transversely (for example a cotangent fibre based at a point outside the image of $S^n\times I^2\subset S^{n+2}$). Thus $S^{n+2}_{\alpha}\in \mathcal{L}_{n+2}$.
\end{proof}

\section{Proof of Theorem \ref{thm:exoticform}}
\label{sec:exoticstructure}

By Proposition \ref{prp:smoothisotopy}, $\tau^{-1} \tau_\phi \in \OP{Symp}^c(T^* S^n)$ is isotopic to the identity by a smooth compactly supported isotopy in the case when $n=4\ell+3$. In fact, as Lemma \ref{lma:formsymp} below shows, there is a smooth isotopy
\begin{gather*}
\Phi_t \co T^*S^n \to T^*S^n,\\
\Phi_t = \begin{cases}
\OP{id}_{T^*S^n}, & t \le 0,\\
\phi^*, & t \ge 1,
\end{cases}
\end{gather*}
and hence, the former isotopy can be taken to be $\tau^{-1} \Phi_t  \tau  \Phi_t^{-1}$. By Lemma \ref{lma:param}, we can find a further isotopy from $\tau^{-1}\tau_{\phi}$ to a compactly supported symplectomorphism $C$ which fixes a cotangent fibre $T^*_xS^n$ setwise. Let us write $M_t$ for the concatentation of these two isotopies; the key properties of this isotopy are that
\begin{itemize}
\item each $M_t$ is a compactly supported diffeomorphism of $T^*S^n$,
\item $M_t=\OP{id}$ for $t \le 0$ while $M_t=C$ for $t \ge 1$,
\item $M_t$ has the form $\tau^{-1}\Phi_{2t}\tau\Phi_{2t}^{-1}$ for $t\leq 1/2$ while each $M_t$ is a symplectomorphism for $t \ge 1/2$.
\end{itemize}
We use this isotopy to construct the diffeomorphism
\begin{gather*} \Psi_\phi \co T^*S^n \times T^*S^1 \to T^*S^n \times T^*S^1,\\
(x,q,p) \mapsto (M_p(x),q,p).
\end{gather*}
Observe that the two-form $\Psi_\phi^*(d\theta_{S^n \times S^1})-d\theta_{S^n \times S^1}$ is compactly supported, even though $\Psi_\phi$ is not ($\theta_{S^n\times S^1}$ denotes the canonical 1-form on $T^*(S^n\times S^1)$).
\begin{thm}
Suppose that $n=4\ell+3$, $\ell \ge 1$, and let $\Psi_\phi$ be as above. If $S_{[\phi]}^{n+1} \notin \mathcal{L}_{n+1}$, then there exists no compactly supported symplectomorphism
\[ \Psi \co (T^*(S^n \times S^1),d\theta_{S^n \times S^1}) \to (T^*(S^n \times S^1),\Psi_\phi^*(d\theta_{S^n \times S^1})).\]
Furthermore, $\Psi_\phi^*(d\theta_{S^n \times S^1})$ is homotopic to $d\theta_{S^n \times S^1}$ by a compactly supported homotopy through non-degenerate two-forms.
\end{thm}
\begin{rmk}
In other words, the above theorem shows that the one-parameter version of the $h$-principle for symplectic forms fails. Namely, if the path of non-degenerate two-forms joining $\Psi_\phi^*(d\theta_{S^n \times S^1})$ and $d\theta_{S^n \times S^1}$ could be deformed to a path of symplectic forms relative infinity, an application of Moser's trick would provide a compactly supported symplectomorphism between the two symplectic structures.
\end{rmk}
\begin{proof}
By the construction of $\Psi_\phi$, we may assume that $\Psi_\phi \circ \Psi$ is supported inside $D^*S^n \times T^*S^1$. The symplectomorphism $\Psi_\phi \circ \Psi$ lifts to a symplectomorphism of the universal cover $(T^*(S^n \times \RR),d\theta_{S^n\times \RR})$ supported inside $D^*S^n \times T^*\RR$. Furthermore, there is some $N>0$ such that this symplectomorphism is of the form
\begin{gather*}
\Psi_\phi \circ \Psi|_{\{ p \le -N\}} = (\id_{T^*S^n},\id_{T^*\RR})|_{\{ p \le -N\}},\\
\Psi_\phi \circ \Psi|_{\{ p \ge N\}} =(C,\id_{T^*\RR})|_{\{ p \ge N\}}.
\end{gather*}
For convenience, let $I$ denote the interval $[-N,N]$. The proof is similar to the proof of Theorems \ref{thm:anothersymp} and \ref{thm:maintheorem} ($k=0$ case). Since $C$ preserves $T^*_xS^n$ setwise, the composition of $\Psi_{\phi}\circ\Psi$ with the standard embedding $T^*_xS^n\times T^*_0\RR$ is a new Lagrangian embedding $e\co\RR^n\times \RR\to T^*(S^n\times\RR)$ which agrees with $T^*_xS^n\times T^*_0\RR$ outside a compact set, say $D^*_xS^n \times D^*_0\RR$.

There is a symplectic embedding $\iota\co D^*S^n\to A^n_2$ of a disc subbundle into $A^n_2$ such that $\iota(S^n)=L_2$ and $\iota(D^*_xS^n)=V\subset L_1$. Note that
\[L'=(\iota\times\OP{id})(e(\RR^n\times I)\cap (D^*S^n\times D^*\RR))\]
agrees with $V\times I$ near the boundary.

Then use Proposition \ref{prp:openemb} to embed an open neighbourhood of $(L_1\cup L_2)\times I\subset A^n_2\times T^*I$ into $T^*S^{n+1}$ so that $L_1\times I$ is mapped into a subset of $S^{n+1}$. Excising $V\times I\subset S^{n+1}$ and replacing it with $L'$ gives a Lagrangian $S^{n+1}_{\alpha}\subset T^*S^{n+1}$ which intersects a cotangent fibre once transversely (for example a cotangent fibre based at a point outside the image of $S^n\times I\subset S^{n+1}$). Thus $S^{n+1}_{\alpha}\in \mathcal{L}_{n+1}$.

To establish the existence of the sought homotopy, we argue as follows. First, note that
\[\omega:=\Phi_\phi^*(d\theta_{S^n \times S^1})=\pi_{T^*S^n}^*((\tau^{-1}\Phi_p  \tau  \Phi_p^{-1})^* (d\theta_{S^n}))-dq \wedge dp+\beta(p)\wedge dp,\]
where $\pi_{T^*S^n} \colon T^*(S^n \times S^1) \to T^*S^n$ is the natural projection, and where $\beta(p)$ is a family of compactly supported one-forms on $T^*S^n$ satisfying $\beta(p)=0$ whenever $p \le 0$ or $p \ge 1$. The two-form
\[\omega':=\pi_{T^*S^n}^*((\tau^{-1}\Phi_p  \tau  \Phi_p^{-1})^* (d\theta_{S^n}))-dq \wedge dp\]
is non-degenerate, and coincides with $\omega$ outside of a compact set. It can be explicitly checked that a linear interpolation $(1-s)\omega+s\omega'$ is a compactly supported homotopy through non-degenerate two-forms.

It now remains to prove that $\omega'$ is homotopic to $d\theta_{S^n \times S^1}$ by a compactly supported homotopy through non-degenerate two-forms. We will show that the path
\[(\tau^{-1} \Phi_t \tau \Phi_t^{-1})^*(d\theta_{S^n}), \:\: t \in [0,1],\]
of symplectic forms on $T^*S^n$, which are standard outside of a compact set, is homotopic through non-degenerate two-forms to $d\theta_{S^n}$ relative endpoints $t \in \{0,1\}$.

Write $\eta_t :=\Phi_t^*(d\theta_{S^n})$. Lemma \ref{lma:formsymp} below produces a (non compactly supported) homotopy $\eta_{t,s}$ of  paths relative endpoints of non-degenerate two-forms satisfying
\begin{gather*}
\eta_{t,0}=\eta_t ,\:\:\eta_{t,1}=d\theta_{S^n}.
\end{gather*}
Using a smooth compactly supported bump-function $\rho \colon T^*S^n \to \RR$ which is equal to one in some neighbourhood of the zero-section, we can construct a compactly supported homotopy
\[\widetilde{\eta}_{t,s} :=\eta_{t,\rho s}, \:\:\widetilde{\eta}_{t,0}=\eta_t.\]
This homotopy may thus be assumed to satisfy $\widetilde{\eta}_{t,s}=\eta_{t,s}$ on any given compact neighbourhood of the zero-section, while $\widetilde{\eta}_{t,s}=\eta_t$ holds outside of some bigger compact set.

The sought compactly supported homotopy contracting the loop
\[(\tau^{-1} \Phi_t \tau \Phi_t^{-1})^*(d\theta_{S^n})=(\Phi_t^{-1})^*((\tau^{-1})^*(\Phi_t^*(d\theta_{S^n})))=(\Phi_t^{-1})^*((\tau^{-1})^*(\eta_t))=(\Phi_t^{-1})^*((\tau^{-1})^*(\widetilde{\eta}_{t,0}))\]
can now be constructed as follows. First, there is a compactly supported homotopy to the loop $(\Phi_t^{-1})^*((\tau^{-1})^*(\widetilde{\eta}_{t,1}))$. Since the support of $\tau^{-1}$ may be assumed to be contained in the set where $\widetilde{\eta}_{t,1}=\eta_{t,1}=d\theta_{S^n}$, and since $\tau^{-1}$ is a symplectomorphism, the latter loop may be assumed to be of the form $(\Phi_t^{-1})^*(\widetilde{\eta}_{t,1})$. This loop is in turn homotopic to \[(\Phi_t^{-1})^*(\widetilde{\eta}_{t,0})=(\Phi_t^{-1})^*(\eta_t)=(\Phi_t^{-1})^*(\Phi_t^*(d\theta_{S^n}))=d\theta_{S^n}.\]
Observe that these homotopies are through non-degenerate two forms and have compact support.
\end{proof}

\begin{lma}
\label{lma:formsymp}
Suppose that a parametrisation $\ell \co L \to T^*L$ of the zero-section is formally Lagrangian isotopic to the canonical inclusion of the zero-section. It follows that there exists a (non compactly supported) isotopy $\Phi_t \co T^*L \to T^*L$ satisfying the property that $\Phi_0=\ell^*$, while $\Phi_1 =\OP{id}_{T^*L}$. The corresponding path of diffeomorphisms $\Phi_t \co T^*L \to T^*L$ may furthermore be taken to satisfy the property that the path of non-degenerate two-forms $\eta_t:=\Phi_t^*(d\theta_L)$ is homotopic (through non-degenerate two-forms) to $d\theta_L$ relative endpoints.
\end{lma}
\begin{proof}
We start by constructing $\Phi_t$. By assumption, there exists a compactly supported smooth isotopy
\[F_t \co T^*L \to T^*L,\:\: t\in[0,1/2],\]
such that $F_0=\id$, $F_{1/2} \circ \ell$ is the canonical inclusion, and where $F_{t/2} \circ \ell$ is the formally Lagrangian isotopy.

Consider the smooth isotopy
\[\Phi_t:=F_t \circ \ell^* \co T^*L \to T^*L,\:\: t \in [0,1/2],\]
where hence $\Phi_0 = \ell^*$. The fact that the isotopy is formally Lagrangian implies that we may assume that $\Phi_{1/2}=\OP{id}_{T^*L}$ holds in some neighbourhood of the zero-section. The Alexander trick can now be used to show that $\Phi_{1/2}$ is smoothly isotopic to $\OP{id}_{T^*L}$ by an isotopy supported in the complement of some neighbourhood of the zero section: simply consider the family $r_{1/t} \circ \Phi_{1/2} \circ r_t$, $t \in (0,1]$, where $r_t$ denotes fibre-wise multiplication by $t \in \RR$. By concatenation, we can thus extend the isotopy $\Phi_t$, $t \in [0,1/2]$, to an isotopy defined for $t \in [0,1]$, and for which $\Phi_1=\OP{id}_{T^*L}$ holds on all of $T^*L$.

It remains to show that the loop $\eta_t:=\Phi_t^*(d\theta_L)$ of symplectic forms satisfies the sought properties.

The fact that the isotopy is formally Lagrangian implies that there exists a path of bundle-automorphisms
\[A_t \co TT^*L|_L \to TT^*L|_L\]
covering the identity, for which
\[\Phi_t^*(d\theta_L) \circ A_t|_L=d\theta_L|_L,\]
and where $A_t$ moreover is homotopic to the identity relative endpoints. We will use $A_{t,s}$ to denote this homotopy, where \begin{gather*}
A_{t,0}=\OP{id}_{TT^*L|_L},\:\:A_{t,1}=A_t.
\end{gather*}

Choosing a metric on $L$, the induced Levi-Civita connection on the bundle $T^*L$ induces an identification $TT^*L \simeq V \oplus H$, where $H$ and $V$ denotes the so-called horizontal and vertical sub-bundles, respectively. Let $\pi \co T^*L \to L$ denote the canonical projection. Recall that there is a canonical identification of $V_x$ with $T_{\pi(x)}^*L$, as well as of $H_x$ with $T_{\pi(x)}L$, where the latter identification is induced by the tangent map $D\pi \co TT^*L \to TL$.

It is a standard fact that both $V$ and $H$ are Lagrangian subbundles of $TT^*L$. For $V$ this is obvious, while for $H$ the claim can be deduced from the well-known fact that the above choice of Rimannian metric on $L$ determines a compatible almost complex structure $J$ on $T^*L$ satisfying $H=JV$ (see {\cite[Section 5, Appendix (iii)]{Dombrowski}}), since a compatible almost complex structure preserves the symplectic form by definition. Recall that the almost complex structure $J$ is determined by the requirement that for $v\in V_x\cong T^*_{\pi(x)}L$ identified with a vector $v'\in T_{\pi(x)}L$ using the metric, $J(v)$ is defined to be the horizontal lift of $v'$ to $H_x$. 

Using $r_s \co T^*L \to T^*L$ to denote the bundle morphism which rescales each fibre via scalar multiplication by $s$,
there exists a smooth family of bundle morphisms
\[
\xymatrix{ TT^*L \ar[d]^{\pi} \ar[r]^{R_s} &  TT^*L\ar[d]^{\pi} \\
T^*L \ar[r]^{r_s} & T^*L,
}\]
covering $r_s$ and which is uniquely determined by the requirement that $R_s$ commutes with $J$, preserves $V$, and that the restriction $R_s:V_x \to V_{sx}$ moreover is induced by the canonical identifications $V_x \simeq T^*_{\pi(x)}L \simeq V_{sx}$. It follows that $R_s$ satisfies
\begin{gather*}
d\theta_L \circ R_s =d\theta_L, \\
R_1 = \OP{id}_{TT^*L},\:\:R_0 \co TT^*L \to TT^*L|_L.
\end{gather*}
Observe that $R_s$ is not too far from the bundle morphism $Dr_s$, which however does not preserve the symplectic form (except in the case $s=1$), due to the fact that it rescales the vertical subbundle $V$.

The sought homotopy from the loop $\Phi_t^*(d\theta_L)$ to the constant loop $d\theta_L$ is now obtained by concatenating the path
\[ \Phi_t^*(d\theta_L) \circ R_{1-s}, \:\: s \in [0,1],\]
with the path
\[ \Phi_t^*(d\theta_L) \circ A_{t,s} \circ R_0, \:\: s \in [0,1].\]
\end{proof}

\section{Open questions}
\label{sec:opnqu}
There is a Lagrangian embedding $S^n \times S^1 \hookrightarrow \CC^{n+1}$ which, using Weinstein's Lagrangian neighbourhood theorem, induces a symplectic embedding $D^*(S^n\times S^1) \hookrightarrow \CC^{n+1}$ of a co-disc bundle of appropriate radius.
\begin{quest} The compactly supported symplectomorphism $\Psi$ of $T^*(S^n\times S^1)$ defined in Section \ref{sec:symp} induces a compactly supported symplectomorphism of $\CC^{n+1}$ by using the above inclusion. Is this symplectomorphism isotopic to the identity through compactly supported symplectomorphisms? Since any Lagrangian can be disjoined from the support of $\Psi$ by translation, $\Psi$ cannot act nontrivially on the isotopy classes of parametrised Lagrangians, so our methods cannot detect nontriviality of $\Psi$ (this is true more generally in any subcritical Stein manifold where any compact subset can be displaced).
\end{quest}
\begin{quest}
Again, using the above symplectic inclusion, the symplectic structure on $T^*(S^n \times S^1)$ constructed in Section \ref{sec:exoticstructure} when $n=4\ell+3$, $\ell \ge 1$, induces a symplectic structure on $\CC^{n+1}$ which is standard outside of a compact set. Is this symplectic structure symplectomorphic to the standard one by a compactly supported diffeomorphism?
\end{quest}

\begin{quest}
Do the symplectomorphisms $\tau^{-1}\tau_{\phi}$ have infinite order in the symplectic mapping class group $\pi_0(\OP{Symp}^c(T^*S^n))$? Since there are only finitely many diffeomorphism classes of homotopy spheres in each dimension, our methods cannot detect infinite order.
\end{quest}
\begin{quest}
Are the symplectomorphisms $\tau_{\phi}$ and $\tau$ of $T^*S^n$ always smoothly isotopic?
\end{quest}

\begin{acknowledgements}

The authors would like to thank Ivan Smith and Oscar Randal-Williams for helpful discussions, in particular for pointing out the Gromoll filtration and Corollary \ref{cor:browder}.

\end{acknowledgements}

\bibliographystyle{plain}
\bibliography{exotic-diffeo-jtop}

\affiliationone{
  Georgios Dimitroglou Rizell\\
  Centre for Mathematical Sciences\\
  University of Cambridge\\
  Cambridge, CB3 0WB\\
  United Kingdom\\
  \email{g.dimitroglou@maths.cam.ac.uk}}
\affiliationtwo{
  Jonathan David Evans\\
  Department of Mathematics\\
  University College London\\
  Gower Street\\
  London, WC1E 6BT\\
  United Kingdom\\
  \email{j.d.evans@ucl.ac.uk}}
\end{document}